\documentclass[noamsfonts]{amsart}
\allowdisplaybreaks
\newcommand{\bm}[1]{#1}
\usepackage{setspace}
\usepackage[charter,expert]{mathdesign}
\usepackage[osf]{XCharter}
\usepackage{enumitem}
\selectfont
\usepackage[activate={true,nocompatibility},final]{microtype}

\usepackage[all, arc]{xy}
\SelectTips{eu}{}
\entrymodifiers={+!!<0pt,\fontdimen22\textfont2>}

\usepackage[pdftex, colorlinks=true, linkcolor=blue, citecolor=magenta, linktocpage]{hyperref}

\swapnumbers
\theoremstyle{plain}
\newtheorem{thm}[subsubsection]{Theorem}
\newtheorem{lemma}[subsubsection]{Lemma}
\newtheorem{prop}[subsubsection]{Proposition}
\newtheorem{cor}[subsubsection]{Corollary}

\theoremstyle{definition}

\newtheorem{remark}[subsubsection]{Remark}

\theoremstyle{definition}

\numberwithin{equation}{subsubsection}
\renewcommand\theequation{\bfseries\thesubsubsection.\arabic{equation}}

\def\cA{\mathcal{A}}
\def\cB{\mathcal{B}}
\def\cC{\mathcal{C}}
\def\cD{\mathcal{D}}

\def\cL{\mathcal{L}}

\def\cO{\mathcal{O}}

\def\cT{\mathcal{T}}

\def\11{\mathbf{1}}

\def\AA{\mathbb{A}} 
 
\def\CC{\mathbb{C}} 
\def\DD{\mathbf{D}} 
\def\EE{\mathbf{E}} 
 
\def\GG{\mathbb{G}}

\def\PP{\mathbb{P}}

\def\ZZ{\mathbb{Z}}

\def\fn{\mathfrak{n}}

\def\dim{\mathrm{dim}}

\def\Gr{\mathrm{Gr}}

\def\Hom{\mathrm{Hom}}

\def\id{\mathrm{id}}

\def\pH{\vphantom{H}^p\!H}

\def\Spec{\mathrm{Spec}}

\def\real{\mathrm{real}}

\def\Dhol{D^b_h}

\newcommand{\mapright}[1]{\xrightarrow{#1}}
\newcommand{\mapleft}[1]{\xleftarrow{#1}}

\newcommand{\FT}{\mathbf{FT}}
\newcommand{\const}[1]{\underline{\Lambda}_{#1}}

\newcommand{\plusotimes}{\overset{+}{\otimes}}

\makeatletter
\newcommand{\holim@}[2]{%
      \vtop{\m@th\ialign{##\cr
                  \hfil$#1\operator@font holim\,$\hfil\cr
                      \noalign{\nointerlineskip\kern1.5\ex@}#2\cr
                          \noalign{\nointerlineskip\kern-\ex@}\cr}}%
                      }
                      \newcommand{\holim}{%
                            \mathop{\mathpalette\holim@{\leftarrowfill@\textstyle}}\nmlimits@
                        }
                        \makeatother

                        \title{Realizations of exponential sheaves and Fourier transform}
\author{R. Virk}
\begin{document}
\maketitle
\tableofcontents
\renewcommand{\thesubsection}{\textbf{\arabic{subsection}}}
\renewcommand{\thesubsubsection}{\textbf{\arabic{subsection}.\arabic{subsubsection}}}

\subsection{Introduction}
\subsubsection{}
We study realizations of exponential sheaves.
Although we do this in the de Rham setting of Hodge modules and D-modules, there is a parallel story for $\ell$-adic sheaves over finite fields.\footnote{The formal setup is the same. The geometric inputs are proved by different arguments there, sometimes easier and sometimes harder, and there is no ready-made category of tame objects closed under Grothendieck's operations, owing to ramification.}
The categories of exponential sheaves $E(X)$ that we study are N. Katz's ``over-world'' \cite[Fourth Lecture]{Klectures}, equipped with the standard Grothendieck operations (attributed to R. Pink). Setting these up (\S\ref{s:expobjects}) is routine,\footnote{Similar formalism is developed independently, for motivic sheaves, in \cite{GL} and \cite{CHS}. I thank G. Ribeiro for the references.} and `realizations'\footnote{I.e., functors to familiar classical derived categories.} are straightforward to write down. The motivation comes from the cross-characteristic ``analogies'' between exponential sums \cite{K}. The sheaves incarnating these sums are irregular/wild, so Riemann-Hilbert does not aid in matching them in different characteristics via the usual method of spreading out. This is exactly the matching one would like, and the categories $E(X)$ offer a tentative solution.

\subsubsection{}
The objects of $E(X)$ are \emph{classes} of objects of $D(X\times\GG_a)$, the derived category of \emph{tame} objects (regular holonomic D-modules, or mixed Hodge modules). So $E(X)$ is built entirely from data that transfers via the Riemann-Hilbert correspondence. A class $K$ is realized as a classical object: each $\lambda\in\CC^\times$ yields a \emph{realization} functor $\real_\lambda\colon E(X)\to\Dhol(X)$ (see \S\ref{s:realizations}), sending $K$ to $\pi_{X!}(K\otimes\mu^*e^{-\lambda t})\in\Dhol(X)$. This is a holonomic but generally \emph{irregular} D-module owing to the exponential twist. Here $\pi_X$ and $\mu$ are the projections from $X\times\GG_a$, and $e^{-\lambda t}$ is the standard exponential D-module $f'+\lambda f=0$ (an Artin-Schreier sheaf in the parallel story in characteristic $p$, whence exponential sums through the trace formula).
The categories $E(X)$ carry a canonical `exponential kernel' $\EE$ (see \S\ref{s:expobject}) that exists even where no classical exponential sheaf does.\footnote{There is no single Artin-Schreier sheaf valid in all characteristics, and the complex exponential is irregular, underlying no mixed Hodge module and no Betti local system on $\AA^1(\CC)$.}
Under the realization $\real_{\lambda}$ it recovers the exponential D-module $e^{-\lambda t}$.
So $E(X)$ is a category built out of tame objects, while the objects one is really after, such as exponential sums and the cohomology of $e^f$, are irregular/wild and arise as the output of a realization. The question is what data about the irregular realizations can be read off $E(X)$.

\subsubsection{}
We show that the realizations $\real_\lambda$ are t-exact (Theorem \ref{texact}) and faithful on the naturally defined heart $EM(X)\subset E(X)$ (Theorem \ref{faith}). Both statements rest on an analysis of the partial Fourier transform applied to regular input: t-exactness comes down to single-degree concentration, for regular $K$, of the costalk $i_\lambda^! K^\wedge[1]$ (Proposition \ref{prop:weaknonchar}), and faithfulness to a conservativity statement (Proposition \ref{prop:conservative}). Neither holds for irregular/wild $K$, whose Fourier transform picks up singularities from nonzero slopes/breaks at infinity. In other words, neither statement is a formal consequence of the usual yoga of Grothendieck's operations.

\subsubsection{}
Over a point, $EM(\Spec(\CC))$ is the abelian category of exponential mixed Hodge structures of \cite[Chapter 4]{KS} (\S\ref{s:emsheart}). For general $X$, t-exactness and faithfulness let one read $E(X)$, with its operations, as a relative version of that picture over an arbitrary base. In particular, $EM(X)$ is built from mixed Hodge modules and so carries a weight filtration, which the realizations carry over to a notion of weights on the irregular D-modules in their image (\S\ref{s:weights}).

\subsubsection{}
The realizations also commute with $f^*$, $f_!$, $f^!$, $f_*$ and with Verdier duality (Theorem \ref{thm:commutesoperations}). The cases of $f^*$ and $f_!$ are trivial. A dual construction of realization, with $\pi_{X*}$ in place of $\pi_{X!}$, would make $f^!$ and $f_*$ trivial instead. The miracle is that the two constructions agree. The proof rests on the forget supports theorem (Theorem \ref{thm:forgetsupports}).

\subsubsection{}
Separately, on the categories $E(V)$ for $V$ a vector bundle, there is a `universal' Fourier transform. The canonical kernel $\EE$ satisfies an orthogonality property (Proposition \ref{orthogonality}), which uses the defining quotient property of $E(X)$ in an essential way and fails in $D(V\times\GG_a)$. Consequently, the transform is invertible (Theorem \ref{fourthm}) and satisfies the Fourier miracle, commuting with Verdier duality up to a twist (Theorem \ref{thm:fourmiracle}).\footnote{This is again a forget supports statement, like the duality commutation above. For the Fourier transform, duality also follows formally from inversion; the realizations are not invertible, so no such formal shortcut exists.} Under realization it recovers classical Fourier transforms (Theorem \ref{fouriercompatible}). Since realizations carry a notion of weight (\S\ref{s:weights}), purity makes sense even for these irregular objects, which a priori have none. The Fourier transform preserves the class of pure objects (Theorem \ref{thm:fourwts}).

\subsubsection{}
Appendix \ref{s:interlude} and Appendix \ref{s:forgetsupports} are self-contained and may be read independently of the rest of the note. They are written entirely in the classical language of holonomic D-modules (the categories $E(X)$ do not appear).

Appendix \ref{s:interlude} collects standard results on the partial Fourier transform used in the main body, and proves Proposition \ref{prop:weaknonchar}, the key non-formal input into the aforementioned t-exactness of realization.

Appendix \ref{s:forgetsupports} gives an independent proof of a stronger form of Proposition \ref{prop:weaknonchar}.
Although the main text needs it only for the compatibility of realization with duality (Theorem \ref{thm:commutesoperations}), it is expected to be a necessary input for any cross-characteristic comparison. 

\subsection{Conventions}
\subsubsection{}`Map' and `morphism' will be used interchangeably. `Canonical map' will be used as a synonym for `natural transformation of functors'.

\subsubsection{}For a sheaf $K$ on a topological space $Y$, we write $\Gamma(Y, K)$ or sometimes simply $\Gamma(K)$, when the space is clear, for the global sections of $K$.

\subsubsection{}
Most of the constructions below will be made simultaneously for holonomic D-modules with regular singularities and mixed Hodge modules. So for a variety $X$ (i.e., a separated scheme of finite type over $\Spec(\CC)$), we let $D(X)$ denote either:
\begin{enumerate}
\item $D^b(MHM(X))$ - the bounded derived category of algebraic mixed Hodge modules on $X$ (see \cite{S});
\item $D^b_{rh}(X)$ - the bounded derived category of algebraic D-modules on $X$ with regular holonomic cohomology modules.
\end{enumerate}
In the Hodge module setting, $(r)$ denotes the $r$-th Tate twist. For D-modules it is simply the identity functor.

\subsubsection{}
We write $M(X)\subset D(X)$ for the heart of the perverse t-structure on $D(X)$. I.e., $M(X)$ is the abelian category of mixed Hodge modules in the Hodge setting, and the category of regular holonomic D-modules in the non-mixed setting.

\subsubsection{}We write $D^b_h(X)$ for the bounded derived category of holonomic, not necessarily regular, D-modules on $X$. For $X$ smooth, the sheaf of differential operators on $X$ is denoted $\cD_X$. 

\subsubsection{}
Functors on derived categories are always derived: we write $f_*$ instead of $Rf_*$, etc. Verdier duality is denoted $\DD$.
Our functors $f^*, f^!, f_*, f_!$ on D-modules are normalized to commute with their counterparts under the Riemann-Hilbert correspondence, dictated by compatibility with the functors on mixed Hodge modules. Here is the translation from our notation to that of the textbook \cite{HTT} and the references \cite[\S7]{KL} and \cite{K}:
\begin{center}
\begin{tabular}{c|c|c}
Our notation & In \cite{HTT} & In \cite[\S7]{KL}, \cite{K} \\
\hline
$\DD$ & $\mathbb{D}$ & $\mathbb{D}$ \\
$f^!$ & $f^{\dagger}$ & $f^!$ \\
$f^*$ & $f^{\star} = \mathbb{D} f^{\dagger}\mathbb{D}$ & $f^* = \mathbb{D} f^! \mathbb{D}$ \\
$f_*$ & $\int_f$ & $f_*$ \\
$f_!$ & $\int_{f!} = \mathbb{D} \int_f \mathbb{D}$ & $f_! = \mathbb{D} f_* \mathbb{D}$ \\
$K_1\boxtimes K_2$ & $K_1\boxtimes K_2$ & $K_1 \times K_2$ \\
$K_1\otimes K_2$ & $\Delta^{\star}(K_1\boxtimes K_2)$ & $\Delta^*(K_1\times K_2)$
\end{tabular}
\end{center}

\subsubsection{}
The constant sheaf in $D(X)$ (i.e., the monoidal unit for $\otimes$) will be denoted $\const{X}$. When $X$ is smooth and connected, the D-module underlying $\const{X}[\dim(X)]$ is the structure sheaf $\cO_X$.

\subsection{Exponential objects}\label{s:expobjects}
\subsubsection{}
We start with an elementary general result that allows us to sidestep matching various morphisms across Verdier duality, proper base change, projection formula, etc. (see the proof of Lemma \ref{lem:adj} for a typical application).
\begin{prop}\label{prop:unitprop}
Let $F\colon \cA \to \cB$ and $G\colon \cB \to \cA$ be functors between categories $\cA$ and $\cB$, such that $F$ is left adjoint to $G$. If there exists \emph{any} natural isomorphism $\alpha\colon \id\mapright{\sim} GF$, then the unit of adjunction $\eta\colon \id \to GF$ is an isomorphism.
\end{prop}

\begin{proof}Write $\varepsilon\colon FG\to \id$ for the counit of adjunction. Note that the composition:
\begin{equation}\label{eq:unit}
G\mapright{\eta G} GFG \mapright{G \varepsilon } G
\end{equation}
is the identity. Furthermore, by the naturality of $\alpha$, the following square commutes:
\[\xymatrix{
\id\ar[d]^{\alpha} \ar[r]^{\alpha} & GF \ar[d]^{GF\alpha} \\
GF \ar[r]^{\alpha GF} & GFGF
}\]
Since $\alpha$ is invertible, this yields:
\begin{equation}\label{eq:alpha}
\alpha GF = GF \alpha.
\end{equation}
Let $s\colon GF\to\id$ be the composition:
\[ GF \mapright{GF\alpha} GFGF \mapright{G\varepsilon F}GF \mapright{\alpha^{-1}} \id.\]
We will show that $s$ is an inverse to $\eta$.

The map $s\circ \eta$ is the composition:
\[ \id \mapright{\eta} GF \mapright{GF\alpha} GFGF \mapright{G\varepsilon F}GF \mapright{\alpha^{-1}} \id.\]
Naturality of $\eta$ implies that this coincides with:
\[ \id \mapright{\alpha} GF \mapright{\eta GF} GFGF \mapright{G\varepsilon F}GF \mapright{\alpha^{-1}} \id. \]
Now apply \eqref{eq:unit} to obtain that this is the identity.

On the other hand, $\eta\circ s$ is the composition:
\[ GF \mapright{GF\alpha} GFGF \mapright{G\varepsilon F}GF \mapright{\alpha^{-1}} \id \mapright{\eta} GF.\]
Naturality of $\eta$ implies that this is none other than:
\[ GF \mapright{GF\alpha} GFGF\mapright{G\varepsilon F}GF\mapright{\eta GF}GFGF\mapright{GF\alpha^{-1}}GF.\]
Using naturality of $\eta$ again, this becomes the composition:
\[ GF\mapright{GF\alpha}GFGF\mapright{\eta GFGF}GFGFGF\mapright{GFG\varepsilon F}GFGF\mapright{GF\alpha^{-1}} GF. \]
Another application of the naturality of $\eta$ makes this:
\[ GF \mapright{\eta GF}GFGF \mapright{GFGF\alpha}GFGFGF\mapright{GFG\varepsilon F}GFGF\mapright{GF\alpha^{-1}}GF.\]
Using \eqref{eq:alpha}, this becomes:
\[GF \mapright{\eta GF}GFGF \mapright{\alpha GFGF}GFGFGF\mapright{GFG\varepsilon F}GFGF\mapright{\alpha^{-1}GF}GF.\]
Naturality of $\alpha$ means this is just the composition:
\[ GF \mapright{\eta GF}GFGF\mapright{G\varepsilon F}GF\mapright{\alpha GF}GFGF\mapright{\alpha^{-1} GF}GF.\]
Now using \eqref{eq:unit} finally gives us that this is the identity on $GF$.
\end{proof}

\subsubsection{}\label{s:def}
For a variety $X$, we write:
\[ \pi_X\colon X \times \GG_a \to X \]
for the projection map.

\begin{lemma}\label{lem:adj}
For $L\in D(X)$, the unit of adjunction $\eta_L\colon L \to \pi_{X*}\pi_X^*L$ is an isomorphism.
In particular, the functor $\pi_X^*\colon D(X) \to D(X\times \GG_a)$ is full and faithful.
\end{lemma}

\begin{proof}By Proposition \ref{prop:unitprop}, it suffices to construct a canonical isomorphism $\id \to \pi_{X*}\pi_X^*$. Using Verdier duality, it is sufficient to construct a canonical isomorphism $\pi_{X!}\pi_X^! \to \id$. Now $\pi_X^! \simeq \pi_X^*[2](1)$, and using the projection formula reduces us to the isomorphism $\pi_{X!}\const{X\times \GG_a}[2](1) \simeq \const{X}$.
\end{proof}

\begin{prop}\label{prop:counitcharacterization}
An object $K\in D(X\times \GG_a)$ is in the essential image of 
\[ \pi_X^*\colon D(X) \to D(X\times\GG_a) \]
if and only if the counit of adjunction $\varepsilon_{K}\colon \pi_X^*\pi_{X*}K \to K$ is an isomorphism.
\end{prop}

\begin{proof}
If the counit is an isomorphism, then it is clear that $K$ is in the essential image. For the converse we need to show that $\varepsilon_{\pi_X^*L}\colon \pi_X^*\pi_{X*}\pi_X^*L \to \pi_X^*L$ is an isomorphism, for each $L\in D(X)$. Write $\eta_L\colon L \to \pi_{X*}\pi_X^*L$ for the unit of adjunction. By Lemma \ref{lem:adj}, we have that $\pi_X^*(\eta_L)$ is an isomorphism. On the other hand, the standard relations between the unit and counit imply that $\varepsilon_{\pi_X^*L}$ is a left inverse to $\pi_X^*(\eta_L)$. Consequently, $\varepsilon_{\pi_X^*L}$ must also be right inverse to $\pi_X^*(\eta_L)$.
\end{proof}

\subsubsection{}Recall that a subcategory $\cC$ of a triangulated category $\cT$ is called \emph{thick} if it is triangulated, and it contains all direct summands of its objects.

\begin{prop}\label{prop:thick}
The essential image of $\pi_X^*\colon D(X) \to D(X\times\GG_a)$ is a thick subcategory.
\end{prop}

\begin{proof}
Apply Proposition \ref{prop:counitcharacterization}.
\end{proof}

\subsubsection{}
Define $E(X)$ to be the Verdier localization (see \cite[Theorem 2.1.8]{Neeman}):
\[ E(X) = D(X\times \GG_a)/\pi_X^*D(X). \]
So $E(X)$ is a triangulated category equipped with a localization functor:
\[ D(X\times\GG_a) \to E(X) \]
and characterized by the following universal property: if $D(X\times \GG_a) \to \cT$ is a triangulated functor whose kernel contains the essential image of $\pi_X^*$, then it factors uniquely as:
\[ D(X\times \GG_a) \to E(X) \to \cT. \]
\begin{prop}\label{prop:kernel}
The localization $D(X\times \GG_a) \to E(X)$ sends $K\in D(X\times\GG_a)$ to an object isomorphic to $0$ in $E(X)$ if and only if $K\in \pi_X^*D(X)$.
\end{prop}

\begin{proof}
The implication that $\pi^*_XD(X)$ gets sent to $0$ is trivial. The other direction is implied by Proposition \ref{prop:thick}. See \cite[Remark 2.1.39]{Neeman}.
\end{proof}
Informally, $E(X)$ is obtained from $D(X\times\GG_a)$ by killing \emph{exactly} all the objects that come from $D(X)$.

\begin{prop}The quotient $Q\colon D(X\times \GG_a)\to E(X)$ admits a full and faithful left adjoint $Q_!$, as well as a full and faithful right adjoint $Q_*$. The inclusion of the kernel of $Q$, i.e., $\pi_X^*\colon D(X) \to D(X\times \GG_a)$, along with its left and right adjoints $\pi_{X!}[2](1)$ and $\pi_{X*}$ is a recollement in the sense of \cite[\S1.4]{BBD}. In particular, for each $K\in D(X\times \GG_a)$ there are canonical distinguished triangles:
\[ Q_!QK \to K \to \pi_X^*\pi_{X!}K[2](1) \to \quad\text{and}\quad \pi_X^*\pi_{X*}K \to K \to Q_*QK\to \]
\end{prop}

\begin{proof}
The functor $\pi_X^*$ is full and faithful (Lemma \ref{lem:adj}) with essential image the thick subcategory $\pi_X^*D(X)$ (Proposition \ref{prop:thick}). A thick subcategory whose inclusion has both left and right adjoints induces, on the Verdier quotient, full and faithful left and right adjoints to the quotient functor, as well as the requisite triangles \cite[Chapter 9]{Neeman}, \cite[\S1.4]{BBD}.
\end{proof}

\begin{cor}The functors $Q_!$ and $Q_*$ identify $E(X)$ with the full subcategories with objects consisting of:
\[\{ K  \mid \pi_{X!}K = 0\} \quad\text{and}\quad\{K  \mid \pi_{X*}K = 0\}, \]
respectively.
\end{cor}

\subsubsection{}\label{s:emsheart}
The Corollary shows that the category of exponential mixed Hodge structures considered in \cite[Chapter 4]{KS} is none other than the abelian heart of our $E(\Spec(\CC))$ (see \S\ref{s:tstructures}).

\subsubsection{}
Given a morphism $f\colon X\to Y$, the functor $(f\times \id_{\GG_a})^*$ induces a functor:
\[ \bm{f^*}\colon E(Y) \to E(X).\]
Proper base change implies that $(f\times \id_{\GG_a})_!$ induces a functor:
\[ \bm{f_!}\colon E(X) \to E(Y).\]

\subsubsection{}\label{s:verdier}
Using $(f\times\id_{\GG_a})_*$ and $(f\times\id_{\GG_a})^!$ one also obtains functors $\bm{f_*}$ and $\bm{f^!}$ on the categories $E(X)$, since $\pi_X$ is smooth. Verdier duality is likewise available. 

\subsubsection{}Write
\[ \text{sum}\colon \GG_a \times \GG_a \to \GG_a \]
for the group operation.
Consider the diagram:
\[ \xymatrix{
& X\times \GG_a \times \GG_a \ar[ld]_{p_{12}}\ar[rd]^{p_{13}}\ar[rr]^{\id_X \times \text{sum}} && X \times \GG_a \\
X \times \GG_a & & X \times \GG_a
}\]
where $p_{12}$ and $p_{13}$ are projections onto the indicated factors.
For $K,L\in D(X\times \GG_a)$, define $K\plusotimes L\in D(X\times\GG_a)$ by:
\[ K \plusotimes L = (\id_X\times\text{sum})_!(p_{12}^*K\otimes p_{13}^*L). \]
This defines a symmetric monoidal structure on $D(X\times \GG_a)$ with unit object given by 
\[ \11_X = (\id_X \times i_0)_!\const{X},\]
where $i_0\colon \Spec(\CC) \to \GG_a$ is the unit for the group operation. Proper base change yields that this induces a symmetric monoidal structure on $E(X)$. 

\subsubsection{}Classical proper base change implies $\bm{f^*}\colon E(Y) \to E(X)$ is $\plusotimes$-monoidal and that $\bm{f_!}\colon E(X) \to E(Y)$ satisfies proper base change. The projection formula with respect to $\plusotimes$ also holds in $E(X)$: this uses both classical proper base change and the classical projection formula. The standard tensor product $\otimes$ does \emph{not} descend to these categories. In short, the usual formalism of Grothendieck's operations holds on the categories $E(X)$, provided one replaces $\otimes$ with $\plusotimes$.

\subsubsection{}The quotient property of $E(X)$ has not been used so far. It will enter in a crucial way in Proposition \ref{orthogonality}. Further, we will not need this and provide no proof, but convolution could equally have been defined using $\text{sum}_*$ instead of $\text{sum}_!$. Although this $\ast$-convolution differs from our $!$-convolution on $D(X\times \GG_a)$, the two become canonically isomorphic in $E(X)$.\footnote{The cone of the forget supports map from $!$-convolution to $\ast$-convolution comes from $D(X)$. See \cite[\S4.4, Lemma 1]{KS} for an argument in the case $X=\Spec(\CC)$.}

\subsection{The exponential kernel}\label{s:expobject}
\subsubsection{}
Let $\Delta\colon \GG_a \to \GG_a\times \GG_a$ be the diagonal morphism. Define $\EE\in E(\GG_a)$ to be the class of:
\[ \EE = \Delta_!\const{\GG_a},\]
where $\Delta_!\colon D(\GG_a) \to D(\GG_a\times \GG_a)$ is the classical pushforward (as opposed to its counterpart on $E(\GG_a)$).
Given morphisms $f,g\colon X \to \GG_a$, set $f+g = \text{sum}\circ (f\times g)$. 

\begin{prop}[Additivity]\label{additivity}
In $E(X)$, we have:
\[ \bm{f^*}\EE \plusotimes \bm{g^*}\EE \simeq \bm{(f+g)^*}\EE.\]
Here $f^*, g^*, (f+g)^*$ are the induced functors $E(\GG_a) \to E(X)$, induced by $(f\times \id_{\GG_a})^*\colon D(\GG_a\times \GG_a) \to D(X\times \GG_a)$, etc.
\end{prop}

\begin{proof}
Work in the categories $D(-)$.
Let $\gamma_g\colon X \to X\times \GG_a$ be the graph $\gamma_g(x) = (x, g(x))$ and let:
\[ \tilde\gamma_f\colon X\times \GG_a \to X\times \GG_a\times \GG_a \]
be given by $\tilde\gamma_f(x,t) = (x, f(x), t)$. Then by classical proper base change:
\[ (f\times \id_{\GG_a})^*\Delta_!\const{\GG_a} \plusotimes (g\times \id_{\GG_a})^*\Delta_!\const{\GG_a} \simeq (\id_X\times\text{sum})_!(\tilde\gamma_{f!}\const{X\times \GG_a} \otimes p_{13}^*\gamma_{g!}\const{X}).\]
By the classical projection formula this is isomorphic to:
\[ (\id_X\times\text{sum})_!\tilde\gamma_{f!}\tilde\gamma_f^*p_{13}^*\gamma_{g!}\const{X}.\]
But $p_{13}\circ \tilde \gamma_f = \id_{X\times \GG_a}$, so this becomes:
\[ (\id_X\times\text{sum})_!\tilde\gamma_{f!}\gamma_{g!}\const{X}.\]
Now $(\id_X\times\text{sum})\circ\tilde\gamma_f\circ\gamma_{g}$ is the graph of $f+g$. So applying proper base change to $((f+g)\times \id_{\GG_a})^*\Delta_!\const{\GG_a}$ gives the desired result.
\end{proof}

\begin{cor}\label{einvertible}
In $E(\GG_a)$, we have:
\[ \EE \plusotimes [-1]^*\EE \simeq \11_{\GG_a}, \]
where $[-1]\colon \GG_a \to \GG_a$ is the negation map. In particular, $\EE$ is monoidally invertible with respect to $\plusotimes$. Note: the functor $[-1]^*$ in the formula above is on $E(\GG_a)$.\end{cor}

\begin{proof}
By additivity, we are reduced to showing $[0]^*\EE \simeq \11_{\GG_a}$, where $[0]\colon \GG_a \to \GG_a$ maps everything to the group unit. This is verified in the category $D(\GG_a\times \GG_a)$ using the definition of $\EE$ and proper base change.
\end{proof}

\subsubsection{}
Let $\pi\colon V\to S$ be a vector bundle of constant rank $r\geq 1$. Let $V^{\vee}\to S$ be the dual bundle, and $m\colon V\times_S V^{\vee} \to \GG_a$ the canonical pairing. Write $q\colon V\times_S V^{\vee}\to V^{\vee}$ for the projection.
The following result is the key step in the invertibility of the Fourier transform (Theorem \ref{fourthm}). The quotient property of $E(X)$ is essential: the statement fails in $D(V^{\vee}\times \GG_a)$. 

\begin{prop}[Orthogonality]\label{orthogonality}
In $E(V^{\vee})$, we have:
\[ \bm{q_!m^*}\EE \simeq \bm{s_!}\11_S[-2r](-r), \]
where $s\colon S \to V^{\vee}$ is the zero section.
Here all the functors are on the categories $E(X)$. I.e., $m^*\colon E(\GG_a) \to E(V\times_S V^{\vee})$ is the functor induced by $(m\times \id_{\GG_a})^*\colon D(\GG_a\times \GG_a) \to D(V\times_S V^{\vee}\times \GG_a)$, etc.
\end{prop}

\begin{proof}For the proof, we work in the categories $D(-)$ and the classical functors on them.
Consider the object 
\[ K=(q\times \id_{\GG_a})_!(m\times\id_{\GG_a})^*\Delta_!\const{\GG_a} \]
in $D(V^{\vee}\times \GG_a)$, and the distinguished triangle:
\[ j_!j^*K \to K \to i_*i^*K \to \]
where $i= s \times \id_{\GG_a}$ and $j$ is the complementary open immersion.
We will show that 
\[ j^*K \simeq \const{(V^{\vee}- S)\times\GG_a}[-2(r-1)](-(r-1)) \]
in $D((V^{\vee}- S)\times \GG_a)$, and that 
\[ i^*K \simeq (\id_S\times i_0)_!\const{S}[-2r](-r),\]
where $i_0\colon\Spec(\CC) \to \GG_a$ is the group unit. This will give us the result, as the \emph{class} of
$j_!\const{(V^{\vee}- S)\times \GG_a}$
in $E(V^{\vee}) = D(V^{\vee}\times\GG_a)/\pi_{V^{\vee}}^*D(V^{\vee})$ is zero.

For $j^*K$, we have a cartesian square:
\[ \xymatrix{
V \times_S (V^{\vee}- S) \times \GG_a \ar[d]^{p}\ar[r]^{\tilde j} & V\times_S V^{\vee}\times\GG_a\ar[d]^{q\times\id_{\GG_a}} \\
(V^{\vee}- S)\times\GG_a \ar[r]^{j} & V^{\vee}\times \GG_a
}\]
where $\tilde j$ is the evident inclusion and $p(x,y,z) = (y,z)$.
So by proper base change:
\[ j^*K \simeq p_!((m\times \id_{\GG_a}) \circ \tilde j)^*\Delta_!\const{\GG_a}.\]
Now apply proper base change again with the cartesian squares:
\[ \xymatrix{
V\times_S (V^{\vee}- S)\ar[r]^{j'}\ar[d]^{\tilde\gamma_{m}}&V\times_S V^{\vee}\ar[r]^{m}\ar[d]^{\gamma_{m}} & \GG_a\ar[d]^{\Delta} \\
V\times_S (V^{\vee}- S)\times\GG_a \ar[r]^{\tilde j} & V\times_S V^{\vee}\times\GG_a\ar[r]^-{m\times\id_{\GG_a}} & \GG_a\times\GG_a 
}\]
where $\gamma_{m}$ is the graph of $m$, the map $\tilde\gamma_{m}$ is its restriction to $V\times_S (V^{\vee}- S)$, and $j'$ is the evident inclusion. This gives
\[ j^*K\simeq p_!\tilde\gamma_{m!}\const{V\times_S (V^{\vee}- S)}.\]
The map 
\[ p\circ \tilde\gamma_{m}\colon V\times_S (V^{\vee}- S) \to (V^{\vee}- S)\times \GG_a \]
is given by $(x,y)\mapsto (y,m(x,y))$.
This is a Zariski locally trivial $\AA^{r-1}$-bundle.\footnote{Instead of the bundle argument at this step, one may also use the usual `algebraic homotopy lemma' here (see \cite[Proposition 1]{So} or \cite[Lemme 5.5]{L}).} Hence, 
\[ p_!\tilde\gamma_{m!} \const{V\times_S (V^{\vee}- S)} \simeq \const{(V^{\vee}- S)\times \GG_a}[-2(r-1)](-(r-1)).\]

For $i^*K$, we have a cartesian square:
\[ \xymatrix{
V\times \GG_a \ar[d]^{\pi\times \id_{\GG_a}} \ar[r]^{\tilde i} & V\times_S V^{\vee}\times \GG_a\ar[d]^{q\times\id_{\GG_a}} \\
S\times \GG_a \ar[r]^i & V^{\vee}\times\GG_a
}\]
where $\tilde i (x,z) = (x,s\pi(x),z)$. Hence, by proper base change:
\[ i^*K \simeq (\pi\times\id_{\GG_a})_!((m\times\id_{\GG_a})\circ \tilde i)^*\Delta_!\const{\GG_a}. \]
Now apply proper base change again with the cartesian squares:
\[ \xymatrix{
V \ar[r]^{i'}\ar[d]^{\id_V\times i_0} & V\times_S V^{\vee} \ar[d]_{\gamma_{m}}\ar[r]^{m} & \GG_a\ar[d]^{\Delta} \\
V\times \GG_a \ar[r]^{\tilde i} & V\times_S V^{\vee}\times \GG_a \ar[r]^-{m\times \id_{\GG_a}} & \GG_a\times \GG_a
}\]
where $i'(x) = (x,s\pi(x))$, and $i_0\colon \Spec(\CC) \to \GG_a$ is the group unit. This gives
\[ (\pi\times\id_{\GG_a})_!((m\times\id_{\GG_a})\circ \tilde i)^*\Delta_!\const{\GG_a} \simeq ((\pi\times\id_{\GG_a})\circ (\id_V \times i_0))_!\const{V}.\]
But $\pi\colon V\to S$ is a rank $r$ vector bundle, so
\[ i^*K \simeq (\id_S \times i_0)_!\const{S}[-2r](-r). \qedhere\]
\end{proof}

\subsection{Fourier transform}\label{s:fourier}
\subsubsection{}\label{s:fourierdia}
Recall the vector bundle $\pi\colon V\to S$ of constant rank $r\geq 1$, its dual $V^{\vee}\to S$, and the pairing
\[ m\colon V\times_S V^{\vee} \to \GG_a. \]
So we have a diagram:
\[ \xymatrix{
& V\times_S V^{\vee} \ar[ld]_p\ar[rd]^q\ar[rr]^m& & \GG_a \\
V& & V^{\vee}
}\]
where $p$ and $q$ are the evident projections.
\subsubsection{}Define $\FT_V\colon E(V) \to E(V^{\vee})$ by:
\[ \FT_V(K) = \bm{q_!}(\bm{p^*}K \plusotimes \bm{m^*}\EE)[r] .\]
Here, again, the functors in the formula are functors on the categories $E(-)$ (i.e., those induced by $p\times\id_{\GG_a}$, etc.).

\begin{thm}\label{fourthm}Let $a\colon V\mapright{\sim}V^{\vee\vee}$ be the isomorphism defined by $a(v) = -m(v, -)$. Then we have a canonical isomorphism in $E(V^{\vee\vee})$:
\[\FT_{V^{\vee}}\circ\FT_{V}(K) \simeq \bm{a_!}K(-r).\]
\end{thm}

\begin{proof}Let $p_1, p_2, p_{12}$, etc., denote the projections from $V\times_S V^{\vee}\times_S V^{\vee\vee}$ to the named factors.
We also have projections:
\[
\xymatrix{
& V\times_S V^{\vee}\ar[ld]_p\ar[rd]^q & & V^{\vee}\times_S V^{\vee\vee}\ar[ld]_{\widetilde p}\ar[rd]^{\widetilde q} & & V\times_S V^{\vee\vee}\ar[ld]_{\bar q}\ar[rd]^{\bar p}\\
V & & V^{\vee} & & V^{\vee\vee} & & V 
} 
\]
and canonical pairings:
\[ m\colon V\times_S V^{\vee} \to \GG_a, \qquad \widetilde m\colon V^{\vee}\times_S V^{\vee\vee}\to \GG_a. \]
These fit into cartesian squares:
\[ \xymatrix{
V\times_S V^{\vee}\times_S V^{\vee\vee}\ar[d]^{p_{13}}\ar[r]^{\alpha} & V^{\vee}\times_S V^{\vee\vee}\ar[d]^{\widetilde q} \\
V\times_S V^{\vee\vee} \ar[r]^{\beta} & V^{\vee\vee}
} \qquad
\xymatrix{
V \ar[r]^{\Delta}\ar[d]^{\pi} & V\times_S V \ar[r]^{\id_V\times a} & V\times_S V^{\vee\vee}\ar[d]^{\beta} \\
S \ar[rr]_{s}& & V^{\vee\vee} }
\]
\[ \xymatrix{ V \times_S V^{\vee} \times_S V^{\vee\vee} \ar[r]^{p_{12}}\ar[d]^{p_{23}} & V\times_S V^{\vee} \ar[d]^{q}\\
V^{\vee}\times_S V^{\vee\vee} \ar[r]^{\widetilde p} & V^{\vee}} \]
where $\Delta$ is the diagonal, $s\colon S \to V^{\vee\vee}$ is the zero section, $\beta(x,z) = z-a(x)$, and $\alpha(x,y,z) = (y,z-a(x))$.
We compute (all functors below are on the categories $E(-)$):
\begin{align*}
\FT_{V^{\vee}}\circ \FT_V(K)
&= \bm{\widetilde q_!}(\bm{\widetilde p^* q_!}(\bm{p^*}K\plusotimes \bm{m^*}\EE)\plusotimes \bm{\widetilde m^*}\EE) [2r] & \text{(by definition)} \\
&\simeq \bm{\widetilde q_!}(\bm{p_{23!}p_{12}^*}(\bm{p^*}K \plusotimes \bm{m^*}\EE) \plusotimes \bm{\widetilde m^*}\EE)[2r] & \text{(proper base change)} \\
&\simeq  \bm{\widetilde q_! p_{23!}}(\bm{p_{12}^*}(\bm{p^*}K\plusotimes \bm{m^*}\EE)\plusotimes \bm{p_{23}^*\widetilde m^*}\EE)[2r] &\text{(projection formula)} \\
&\simeq \bm{\bar q_! p_{13!}}(\bm{p_{12}^*p^*}K \plusotimes \bm{p_{12}^*m^*}\EE \plusotimes \bm{p_{23}^*\widetilde m^*}\EE)[2r] & \text{($\widetilde q \circ p_{23} = \bar q \circ p_{13}$)}\\
&\simeq \bm{\bar q_! p_{13!}}(\bm{p_{13}^*\bar p^*}K \plusotimes \bm{p_{12}^*m^*}\EE \plusotimes \bm{p_{23}^*\widetilde m^*}\EE)[2r] & \text{($p\circ p_{12} =\bar p\circ p_{13}$)} \\
&\simeq \bm{\bar q_! p_{13!}}(\bm{p_{13}^*\bar p^*}K \plusotimes \bm{\alpha^*\widetilde m^*}\EE)[2r]& \text{(additivity)}\\
&\simeq \bm{\bar q_!}(\bm{\bar p^*}K \plusotimes \bm{p_{13!}\alpha^*\widetilde m^*}\EE)[2r] & \text{(projection formula)} \\
&\simeq \bm{\bar q_!}(\bm{\bar p^*}K \plusotimes \bm{\beta^*\widetilde q_!\widetilde m^*}\EE)[2r] & \text{(proper base change)}\\
&\simeq \bm{\bar q_!}(\bm{\bar p^*}K \plusotimes \bm{\beta^*s_!}\11_{S})(-r) & \text{(orthogonality)}\\
&\simeq \bm{\bar q_!}(\bm{\bar p^*}K \plusotimes \bm{(\id_V\times a)_!\Delta_!}\11_{V})(-r) & \text{(proper base change)} \\
&\simeq \bm{\bar q_!(\id_V\times a)_!\Delta_!}(\bm{\Delta^*(\id_V\times a)^*\bar p^*}K \plusotimes \11_V)(-r) & \text{(projection formula)} \\
&\simeq \bm{\bar q_!(\id_V\times a)_!\Delta_!}(K \plusotimes \11_{V})(-r) & \text{($\bar p\circ (\id_V\times a)\circ \Delta = \id$)}\\
&\simeq \bm{a_!}K(-r) & \text{($\bar q\circ (\id_V \times a) \circ\Delta = a$)}.
\end{align*}
\end{proof}

\begin{thm}\label{thm:fourmiracle}
There is a canonical isomorphism of functors $E(V) \to E(V^{\vee})$:
\[ \FT_V(K) \simeq q_*(p^*K\plusotimes m^*\EE)[r].\]
\end{thm}

\begin{proof}
This is a formal consequence of the inversion formula. The argument is due to J.L. Verdier (see \cite[Th\'eor\`eme 4.1]{L}). In more detail, use the notation in the proof of Theorem \ref{fourthm}.
Define $\widetilde a\colon V\times_S V^{\vee}\to V^{\vee}\times_S V^{\vee\vee}$ by $(x,y)\mapsto (y, a(x))$.
Let $L\in E(V^{\vee})$ and $K\in E(V)$.
Then we have canonical isomorphisms (all functors and $\Hom$ are on/in the categories $E(-)$): 
\begin{align*}
\Hom(a^*\FT_{V^{\vee}}(L)(r), K) &= \Hom(a^*\widetilde q_!(\widetilde p^*L\plusotimes \widetilde m^*\EE)[r](r), K) \\
&\simeq \Hom(\widetilde p^*L\plusotimes \widetilde m^*\EE, \widetilde q^!a_*K[-r](-r)) \\
&\simeq \Hom(\widetilde p^*L\plusotimes \widetilde m^*\EE, \widetilde a_* p^!K[-r](-r)) \\
&\simeq \Hom(\widetilde p^*L\plusotimes \widetilde m^*\EE, \widetilde a_* p^*K[r]) \\
&\simeq \Hom(\widetilde a^*(\widetilde p^*L\plusotimes \widetilde m^*\EE),  p^*K[r]) \\
&\simeq \Hom(q^*L \plusotimes m^*[-1]^*\EE, p^*K[r]) \\
&\simeq \Hom(q^*L, p^*K\plusotimes m^*\EE[r]) \\
&\simeq \Hom(L, q_*(p^*K\plusotimes m^*\EE[r])).
\end{align*}
The first isomorphism is given by standard adjunctions, the second is base change, the third is $p^! \simeq p^*[2r](r)$ (as $p$ is smooth of relative dimension $r$), the fourth is a standard adjunction, the fifth is monoidality of $\ast$-pullback combined with compatibility of compositions, the sixth follows from Corollary \ref{einvertible}, and the last one is a standard adjunction.
Thus, $L\mapsto q_*(p^*L\plusotimes m^*\EE)[r]$ is right adjoint to $a^*\FT_{V^{\vee}}(r)$. On the other hand, by Theorem \ref{fourthm}, $\FT_{V}$ is also right adjoint to $a^*\FT_{V^{\vee}}(r)$. Hence, by the uniqueness of adjoints (Yoneda lemma), we obtain the required result.
\end{proof}

\subsubsection{}
Given a map $f\colon X \to \GG_a$, define the isomorphism:
\[ \tau_f\colon X \times \GG_a \to X \times \GG_a, \quad (x,t)\mapsto (x, t + f(x)).\]

\begin{prop}\label{prop:altdescr}Let $K\in D(X \times \GG_a)$. Then we have a canonical isomorphism:
\[ K \plusotimes (f\times \id_{\GG_a})^*\EE \simeq \tau_{f!} K \]
in $D(X \times \GG_a)$.
\end{prop}

\begin{proof}
Define $\tilde\gamma_f\colon X \times \GG_a \to X \times\GG_a\times \GG_a$ by:
\[ \tilde\gamma_f(x,t) = (x,t, f(x)). \]
Then we have cartesian squares:
\[\xymatrix{
X\ar[r]^{f} \ar[d]^{\gamma_f} & \GG_a \ar[d]^{\Delta} \\
X \times \GG_a \ar[r]^{f\times\id_{\GG_a}}& \GG_a \times \GG_a
}
\quad
\xymatrix{
X \times \GG_a \ar[r]^{\text{pr}}\ar[d]^{\tilde \gamma_f} & X \ar[d]^{\gamma_f}\\
X \times \GG_a\times \GG_a \ar[r]^{p_{13}} & X \times \GG_a
}\]
where $\gamma_f$ is the graph of $f$, $\text{pr}$ is the evident projection, and $p_{13}$ is the projection to the first and third factors.
Thus,
\begin{align*}
K\plusotimes (f\times\id_{\GG_a})^*\EE &\simeq \text{sum}_!(p_{12}^*K \otimes p_{13}^*(f\times \id_{\GG_a})^*\Delta_!\const{\GG_a}) & \text{(by definition)} \\
&\simeq \text{sum}_!(p_{12}^*K \otimes p_{13}^*\gamma_{f!}f^* \const{\GG_a}) & \text{(proper base change)} \\
&\simeq \text{sum}_!(p_{12}^*K \otimes \tilde\gamma_{f!}\text{pr}^*\const{X} ) & \text{(proper base change)}\\
&\simeq \text{sum}_!\tilde\gamma_{f!}\tilde\gamma_f^*p_{12}^*K & \text{(projection formula)} \\
&\simeq \text{sum}_!\tilde\gamma_{f!}K & \text{(as $p_{12}\circ \tilde\gamma_f = \id_{X\times \GG_a}$)} \\
&\simeq \tau_{f!} K & \text{(as $\text{sum}\circ \tilde\gamma_f = \tau_f$)}.
\end{align*}
\end{proof}

\begin{prop}\label{prop:Edual}
Let $K\in E(X)$, and let $f\colon X \to \GG_a$ be a morphism. Then we have a canonical isomorphism in $E(X)$:
\[ \DD(K\plusotimes f^*\EE) \simeq (\DD K) \plusotimes f^*\EE.\]
Here $f^*\colon E(\GG_a) \to E(X)$ is the functor induced by $(f\times \id_{\GG_a})^*\colon D(\GG_a \times \GG_a) \to D(X\times \GG_a)$.
\end{prop}

\begin{proof}
Let $\tilde K \in D(X \times \GG_a)$ be a lift of $K$. By Proposition \ref{prop:altdescr}, the object $\tau_{f!}\tilde K$ is a lift of $K\plusotimes f^*\EE$. Hence, $\DD \tau_{f!}\tilde K \simeq \tau_{f!}\DD \tilde K$ is a lift of $\DD(K\plusotimes f^*\EE)$. Applying Proposition \ref{prop:altdescr} again, we have that $\tau_{f!}\DD \tilde K$ is a lift of $(\DD K) \plusotimes f^*\EE$.
\end{proof}

\begin{cor}\label{cor:fourduality}
Let $K\in E(V)$. Then we have a canonical isomorphism:
\[ \DD(\FT_V(K)) \simeq \FT_V(\DD(K))(r).\]
\end{cor}

\begin{proof}
In $E(V^{\vee})$, we have:
\begin{align*}
\DD(q_!(p^*K\plusotimes m^*\EE)[r]) &\simeq q_*\DD(p^*K \plusotimes m^*\EE) [-r] \\
&\simeq q_*(\DD(p^* K) \plusotimes m^*\EE)[-r] & \text{(Proposition \ref{prop:Edual})}\\
&\simeq q_*(p^* \DD K \plusotimes m^*\EE)[r](r) \\
&\simeq \FT_V(\DD(K))(r) & \text{(Theorem \ref{thm:fourmiracle}).}
\end{align*}
\end{proof}

\subsection{Realizations}\label{s:realizations}
\subsubsection{}Recall that $\Dhol(X)$ is the bounded derived category of holonomic (not necessarily regular) D-modules on $X$. For $K\in D(X)$ we abuse notation and denote the underlying object in $\Dhol(X)$ by $K$ also.

\subsubsection{}\label{s:realcommute}
Let $\cL\in \Dhol(\GG_a)$ be such that:
\[ \pi_!\cL = 0,\]
where $\pi\colon \GG_a \to \Spec(\CC)$.
For $K\in D(X \times \GG_a)$, define:
\[ \real_{\cL}(K) = \pi_{X!}(K \otimes \mu^*\cL), \]
where $\mu\colon X\times \GG_a \to \GG_a$ is the projection map. By our assumption on $\cL$ and the classical projection formula, if $K\in \pi_X^*D(X)$, then $\real_{\cL}(K) = 0$. Thus, we get an induced functor:
\[ \real_{\cL}\colon E(X) \to \Dhol(X).\]
Using classical proper base change and the projection formula, one immediately gets canonical isomorphisms:
\[ f^* \circ \real_{\cL} \simeq \real_{\cL} \circ f^* \qquad \text{and} \qquad f_! \circ \real_{\cL} \simeq \real_{\cL} \circ f_!, \]
for any map $f\colon X\to Y$. These statements follow from the corresponding ones for $f\times \id_{\GG_a}$ on $D(X\times \GG_a)$. I.e., they do not use the quotient nature of $E(X)$ apart from the fact that the condition on $\cL$ ensures that $\real_{\cL}$ descends to the quotient.

\subsubsection{}Application of classical proper base change and the projection formula yields:
\begin{equation}\label{eq:realexpkernel}
\real_{\cL}(\EE) \simeq \cL.
\end{equation}

\begin{prop}
\label{realmonoidal}
Let $\cL\in \Dhol(\GG_a)$ be such that
\[ \text{sum}^*\cL \simeq \cL \boxtimes \cL \quad \text{in $\Dhol(\GG_a\times \GG_a)$}. \]
Then for $K_1,K_2\in E(X)$, we have canonical isomorphisms:
\[ \real_{\cL}(K_1\plusotimes K_2) \simeq \real_{\cL}(K_1)\otimes \real_{\cL}(K_2).\]
\end{prop}

\begin{proof}
This is an application of proper base change and the projection formula yet again. Let $p\colon X \times \GG_a \times \GG_a \to X$ and $q\colon X \times \GG_a \times \GG_a \to \GG_a\times\GG_a$ be the evident projections. Then:
\begin{align*}
\real_{\cL}(K_1\plusotimes K_2) &= \pi_{X!}(\text{sum}_!(p_{12}^*K_1 \otimes p_{13}^*K_2)\otimes \mu^*\cL) & \text{(by definition)} \\
&\simeq p_!(p_{12}^*K_1 \otimes p_{13}^*K_2 \otimes q^*\text{sum}^*\cL) & \text{(projection formula)} \\
&\simeq p_!(p_{12}^*K_1\otimes p_{13}^*K_2 \otimes p_{12}^*\mu^*\cL \otimes p_{13}^*\mu^*\cL) & \text{(hypothesis on $\cL$)} \\
&\simeq p_!(p_{12}^*(K_1\otimes \mu^*\cL) \otimes p_{13}^*(K_2\otimes \mu^*\cL)) \\
&\simeq \pi_{X!}(K_1 \otimes \mu^*\cL) \otimes \pi_{X!}(K_2 \otimes \mu^*\cL) & \text{(K\"unneth formula)} \\
&= \real_{\cL}(K_1) \otimes \real_{\cL}(K_2).
\end{align*}
In the projection formula step, we have also used:
\[ \pi_X\circ (\id_X \times \text{sum}) = p, \qquad \mu\circ (\id_X \times \text{sum}) = \text{sum} \circ q. \]
The K\"unneth formula is in turn a formal consequence of proper base change and the projection formula.
\end{proof}

\begin{prop}\label{prop:unital}
Write $i_0\colon \Spec(\CC)\to \GG_a$ for the group unit, and assume $\cL\in \Dhol(\GG_a)$ is such that $i_0^*\cL = \CC$.
Then
$\real_{\cL}(\11_X) \simeq \const{X}$.
\end{prop}

\begin{proof}
Use proper base change and the projection formula yet again.
\end{proof}

\subsubsection{}\label{s:defL}
For $\lambda \in\CC$, define $\cL_{\lambda}\in \Dhol(\GG_a)$ by requiring $\cL_{\lambda}[1]$ to be the D-module on $\GG_a = \Spec(\CC[t])$ given by:
\[ \Gamma(\GG_a, \cL_{\lambda}[1]) = \CC[t], \qquad \partial_t\cdot p(t) = p'(t) - \lambda p(t). \]
I.e., $\cL_{\lambda}[1] = \cD_{\GG_a}/\cD_{\GG_a}(\partial_t + \lambda)$.
The formal adjoint of the differential operator $\partial_t + \lambda$ is the operator $-\partial_t + \lambda$. Consequently (see \cite[\S2.6]{HTT}):
\begin{equation}\label{eq:dualofclambda}
\DD \cL_{\lambda} \simeq \cL_{-\lambda}[2].
\end{equation}
In particular, $\cL_0  = \const{\GG_a}$.

\subsubsection{}
For $\lambda\in \CC^{\times}$, set:
\[ \real_{\lambda} = \real_{\cL_{\lambda}}.\]

\begin{prop}\label{prop:realkernelproperties}
Let $\lambda\in\CC^{\times}$.
The following hold:
\begin{enumerate}
\item $\pi_! \cL_{\lambda} = 0$;
\item $i_0^*\cL_{\lambda} = \CC$;
\item $\text{sum}^*\cL_{\lambda} \simeq \cL_{\lambda} \boxtimes \cL_{\lambda}$.
\end{enumerate}
\end{prop}

\begin{proof}
These are standard and straightforward. Details are left to the reader (see \cite[\S2.10, \S12]{K}). We just illustrate the elementary computational approach via showing (ii). We have:
\[ i_0^*\cL_{\lambda} \simeq \DD (i^!_0 \DD \cL_{\lambda}) \simeq \DD(i_0^!\cL_{-\lambda}[2]). \]
By the definition of $\cL_{-\lambda}$, we have that $\cL_{-\lambda}[2]$ is concentrated in degree $-1$ and given by the differential equation $f'(t)-\lambda f(t) = 0$ there. So $i^!_0\cL_{-\lambda}[2]$ is represented by the complex:
\[ [\CC[t,\partial_t]/(\partial_t-\lambda) \mapright{t} \CC[t,\partial_t]/(\partial_t - \lambda) ] \]
concentrated in degrees $-1$ and $0$. Now $H^{-1}$ of this complex is $0$ and $H^0$ is $\CC$, whence the result.
\end{proof}

\subsubsection{}
One may also write $\real_{\lambda}$ in terms of the partial Fourier transform:
\[ (-)^{\wedge}\colon D^b_h(X\times \GG_a) \to D^b_h(X\times\GG_a) \]
of Appendix \ref{s:interlude}. Let $i_{\lambda}\colon X \to X\times \GG_a$ be the map given by:
\[ i_{\lambda}(x) = (x,\lambda).\]
\begin{prop}\label{prop:wedgerealcor}
Let $K\in E(X)$, and let $\lambda\in\CC^{\times}$. Then we have a canonical isomorphism:
\[ i^*_{\lambda}K^{\wedge}[-1] \simeq \real_{\lambda}(K). \]
\end{prop}

\begin{proof}
Immediate from the definition of $\real_{\lambda}$ and Proposition \ref{prop:wedgepush}.
\end{proof}

\subsubsection{}\label{s:laplace}
In the notation of \S\ref{s:fourierdia}, one has the D-module Fourier transform:
\[ \FT^{\mathrm{dR}}_V\colon \Dhol(V) \to \Dhol(V^{\vee}) \]
given by the formula:
\[ \FT^{\mathrm{dR}}_V(K) = q_!(p^*K \otimes m^* \cL_1)[r].\]

\begin{thm}\label{fouriercompatible}
We have a canonical isomorphism:
\[ \real_{1} \circ \FT_V \simeq \FT^{\mathrm{dR}}_V \circ \real_{1}. \]
\end{thm}

\begin{proof}
We have already argued above that $\real_{1}$ commutes with $\ast$-pullback and $!$-pushforward quite generally. For commutation with $\plusotimes$, we need the prerequisite condition in Proposition \ref{realmonoidal} for $\cL_{1}$. This is given by Proposition \ref{prop:realkernelproperties}(iii). Finally, $\real_1(\EE) = \cL_1$ by \eqref{eq:realexpkernel}.
\end{proof}

\begin{thm}\label{thm:commutesoperations}
Let $\lambda\in \CC^{\times}$. Then $\real_{\lambda}$ commutes with the standard functors $f^*, f_!, f^!, f_*$. Furthermore, one has a canonical isomorphism:
\[ \DD \circ \real_{\lambda} \simeq \real_{-\lambda} \circ \DD.\]
\end{thm}

\begin{proof}The functors $f^*$ and $f_!$ have already been discussed above in \S\ref{s:realcommute}. The claim for $f^!$ and $f_*$ follows from these combined with the statement for $\DD$. The latter is a consequence of Theorem \ref{thm:forgetsupports}.
\end{proof}

\subsection{t-structures}\label{s:tstructures}
\subsubsection{}
Recall that $M(X)$ denotes the heart of the usual (perverse) t-structure on $D(X)$. As $\pi_X$ is smooth of relative dimension $1$, the functor $\pi_X^*[1]$ is t-exact.

\begin{prop}\label{prop:extensions}
The essential image of
\[\pi_X^*[1]\colon M(X) \to M(X\times \GG_a) \]
is closed under subquotients and extensions.
\end{prop}

\begin{proof}
As $\pi_X$ is surjective and smooth of relative dimension $1$ with connected fibres, \cite[Corollaire 4.2.6.2]{BBD} gives closure under subquotients. Closure under extensions follows from Proposition \ref{prop:counitcharacterization}.
\end{proof}

\subsubsection{}
Using induction and Proposition \ref{prop:counitcharacterization}, one sees that the essential image of $\pi_X^*\colon D(X) \to D(X\times \GG_a)$ consists of those objects whose (perverse) cohomology modules lie in the essential image of $\pi_X^*[1]\colon M(X) \to M(X\times \GG_a)$.
Together with Proposition \ref{prop:extensions} and Proposition \ref{prop:kernel} this puts us in the situation of \cite[Proposition 2.1.4]{B}: the t-structure on $D(X\times \GG_a)$ induces a non-degenerate t-structure on $E(X)$.  Write $EM(X)$ for its heart. Then, under the localization $D(X\times \GG_a)\to E(X)$, we may identify $EM(X)$ with the quotient abelian category:
\[ EM(X) = M(X\times \GG_a)/\pi_X^*[1]M(X).\]
The quotient functor $M(X\times \GG_a) \to EM(X)$ is an exact functor of abelian categories.

\begin{prop}\label{prop:conservative}
Fix $\lambda\in \CC^{\times}$. Let $K\in EM(X)$. If $\real_{\lambda}(K) = 0$, then $K=0$ in $EM(X)$.
\end{prop}

\begin{proof}
The forgetful functor from mixed Hodge modules to (regular) D-modules is conservative. Consequently, we may assume $K$ is a regular holonomic D-module.
It follows from \cite[Theorem 2.10.16]{K} (alternatively, see \cite[Chapter XI, Remarque 4.4.5]{Ma}) that the Fourier transform $K^{\wedge}$ has no singularities on $\{x\} \times (\GG_a - \{0\})$, for each $x\in X(\CC)$.\footnote{For a holonomic D-module $K$ on $\GG_a$, only slopes equal to $1$ at $\infty$ contribute to singularities of $K^{\wedge}$ on $\GG_a-\{0\}$. For regular $K$, slopes at $\infty$ are all $0$.}
Thus, by Proposition \ref{prop:wedgerealcor}, if $\real_{\lambda}(K) = 0$, then the support of $K^{\wedge}$ is contained in $X\times \{0\}$. Consequently, by Fourier inversion, $K$ lies in $\pi_X^*D(X)$.
\end{proof}

\begin{thm}\label{texact}Let $\lambda\in \CC^{\times}$. Then $\real_{\lambda}\colon E(X) \to \Dhol(X)$ is t-exact.
\end{thm}

\begin{proof}
It suffices to show that if $K$ is a regular holonomic D-module on $X\times \GG_a$, then $\real_{\lambda}(K)$ is concentrated in degree $0$, for each $\lambda\in\CC^{\times}$. For this, it in turn suffices to show that $\DD (\real_{-\lambda}(\DD K))$ is concentrated in degree $0$ for all $\lambda\in\CC^{\times}$. By Proposition \ref{prop:wedgereal}, this is equivalent to showing $i^!_{\lambda}K^{\wedge}[1]$ is concentrated in degree $0$. This is Proposition \ref{prop:weaknonchar}.
\end{proof}


\begin{thm}\label{faith}Let $\lambda\in \CC^{\times}$. Then
\[ \real_{\lambda}\colon EM(X) \to \mathrm{Hol}(X) \]
is faithful. Here, $\mathrm{Hol}(X)$ is the category of holonomic D-modules on $X$.
\end{thm}

\begin{proof}
Since $\real_{\lambda}$ is t-exact, it suffices to show that it is conservative. This is given by Proposition \ref{prop:conservative}.
\end{proof}

\begin{remark}For $X = \Spec(\CC)$, this recovers a variant of \cite[\S4, Theorem 4]{KS}.
\end{remark}

\begin{remark}Here is a `realization' functor:
\[ \real_{\phi}\colon E(X) \to D(X) \]
which does not fit the pattern of the $\real_{\lambda}$ that were the focus of \S\ref{s:realizations}.
This functor lands in regular objects.
Namely, let $\phi\colon D(X\times \GG_a) \to D(X)$ be the vanishing cycles functor\footnote{Shift conventions are that vanishing cycles are t-exact.} along the projection $X\times\GG_a \to \GG_a$. Clearly, $\phi$ kills objects that come from $D(X)$. Hence, it factors through a functor $\real_{\phi}$ on $E(X)$. This is t-exact and commutes with Verdier duality, smooth pullback and proper pushforward. It also intertwines $\plusotimes$ with $\otimes$ (Thom-Sebastiani). However, $\real_{\phi}$ is not faithful (or even conservative).
\end{remark}

\begin{remark}\label{rem:nearbyrealization}
Here is a second realization functor of a different character. Let $\Psi$ be the nearby cycles functor along the projection $X\times\GG_a \to \GG_a$ (our shift convention for $\Psi$ is that it is t-exact).
Set:
\[ \real_{\Psi}(K) = \Psi(K^\wedge). \]
Informally: 
\[ \real_{\Psi}(K) \simeq \lim_{\lambda \to 0}i_{\lambda}^*K^{\wedge}[-1] \simeq \lim_{\lambda \to 0} \real_{\lambda}(K). \]
The functor $\real_\Psi$ is t-exact and commutes with Verdier duality. It is the natural $\lambda$-independent analogue of the $\real_{\lambda}$, to which it should be isomorphic for each $\lambda\in\CC^{\times}$. Granting this, the monodromy filtration on $\Psi$ should give a geometric interpretation of the weight filtration on the image of realization (\S\ref{s:weights}).
\end{remark}

\subsection{Weights}\label{s:weights}
\subsubsection{}In this section $D(X) = D^b(MHM(X))$ since we need to discuss weights.
Let $W$ denote the weight filtration on mixed Hodge modules. As the image of 
\[ \pi_X^*[1]\colon M(X) \to M(X\times \GG_a) \]
is closed under subquotients (Proposition \ref{prop:extensions}), one immediately gets that $W$ descends to $EM(X)$. Similarly, it follows that morphisms in $EM(X)$ are \emph{strictly} compatible with the (induced) associated graded functors $\Gr^W_i$ on $EM(X)$.

\subsubsection{}Let $\pH^i$ denote the cohomology functors associated to the t-structure on $E(X)$. Define weights for objects $K\in E(X)$ in the evident way: $K$ is said to have weights $\leq n$ (resp. $\geq n$) if $\Gr^W_k\pH^i(K) = 0$ for $k>i+n$ (resp. $k < i+n$). The object $K$ is called pure of weight $n$ if $\Gr_k^W\pH^i(K) = 0$ for $k\neq i+n$.

\subsubsection{}It is immediate from the corresponding statements in $D(X\times \GG_a)$ that $f^*, f_!$ do not raise weights, $f^!, f_*$ do not lower weights, and Verdier duality reverses weights.

\begin{remark}If $K\in M(X\times \GG_a)$ is pure of weight $k$, then necessarily the class of $K$ in $EM(X)$ is pure of weight $k$. However, if an object $L\in EM(X)$ is pure of weight $k$, and $L'\in M(X\times \GG_a)$ is some lift of it, then $L'$ need only be pure of weight $k$ up to factors coming from $M(X)$. This is due to the identification $EM(X) = M(X\times \GG_a)/\pi_X^*[1]M(X)$, in turn ultimately a consequence of Proposition \ref{prop:kernel}.
\end{remark}

\subsubsection{}To address the effect of Fourier transform on weights, we now use the notation and setup of \S\ref{s:fourier}.
\begin{prop}\label{prop:Ewt}
Suppose $K\in E(V\times_S V^{\vee})$ has weights $\leq n$ (resp. $\geq n$). Then so does $K\plusotimes m^*\EE$.
\end{prop}

\begin{proof}
Let $\tilde K$ be a lift of $K$ to $D(V\times_S V^{\vee}\times \GG_a)$. By Proposition \ref{prop:altdescr}, the object $\tau_{m!}\tilde K \in D(V\times_S V^{\vee}\times \GG_a)$ is a lift of $K \plusotimes m^*\EE$, where $\tau_m$ is the isomorphism given by:
\[ \tau_m\colon V\times_S V^{\vee} \times \GG_a \to V\times_S V^{\vee} \times \GG_a, \quad (x,y,t)\mapsto (x,y, t+m(x,y)). \]
Visibly, $\tau_{m!}$ preserves the subcategory of objects that come from $D(V\times_S V^{\vee})$. Hence, if $K$ has weights $\leq n$ (resp. $\geq n$), then so must the class of $\tau_{m!}\tilde K$ in $E(V\times_S V^{\vee})$.
\end{proof}


\begin{thm}\label{thm:fourwts}
Let $V\to S$ be a vector bundle of constant rank $r\geq 1$.
The Fourier transform $\FT_V\colon E(V) \to E(V^{\vee})$ preserves purity. More precisely, if $K\in E(V)$ is pure of weight $k$, then $\FT_V(K)$ is pure of weight $k+r$.
\end{thm}

\begin{proof}
Let $K\in E(V)$ be pure of weight $k$.
It follows from the definition of $\FT_V$, the general interaction of weights with the standard operations given above, and Proposition \ref{prop:Ewt} that $\FT_V(K)$ has weights $\leq k+r$. So $\DD\FT_V(K)$ has weights $\geq -k-r$. But, by Corollary \ref{cor:fourduality}:
\[ \DD\FT_V(K) \simeq \FT_V(\DD K)(r). \]
The right hand side has weights $\leq -k-r$. So $\DD \FT_V(K)$ must be pure of weight $-k-r$. Hence, $\FT_V(K)$ must be pure of weight $k+r$.
\end{proof}

\subsection{Complement: alternative description of $E_{rh}(X)$}
\subsubsection{}In this section we work exclusively in the non-mixed setting. I.e., $D(X) = D^b_{rh}(X)$. To emphasize this we write $E_{rh}(X)$ instead of $E(X)$.

\subsubsection{}In general, realization $E_{rh}(X) \to D^b_h(X)$ fails to be full. This is already visible for $X=\Spec(\CC)$. However, we will show in Theorem \ref{thm:alt1} that $E_{rh}(X)$ can be embedded as a subcategory of $D^b_h(X\times \GG_m)$. This description was motivated by trying to understand \emph{why} the cross-characteristic comparisons made in the special cases of \cite[Chapter 12-14]{K} hold with the ease (relative to the general situation) that they do.

\subsubsection{}Let $j\colon X \times \GG_m \to X \times \GG_a$ be the evident inclusion of varieties and recall that $i_0\colon X \to X \times \GG_a$, $x\mapsto (x,0)$ is the inclusion of the complementary closed subset $X \times \{0\}$.

\begin{lemma}\label{rhlemma}
Let $K\in D^b_{rh}(X\times \GG_a)$. Then $i_0^!K^{\wedge} \in D^b_{rh}(X)$. That is, $i_0^!K^{\wedge}$ has cohomology modules with regular singularities.
\end{lemma}

\begin{proof}
By Proposition \ref{prop:wedgereal}, $i_0^!K^{\wedge} \simeq \pi_{X*}K[-1]$ and $\pi_{X*}$ preserves regularity.
\end{proof}

\subsubsection{}Consider the functor $D^b_{rh}(X\times \GG_a) \to D^b_h(X \times \GG_m)$ given by $K\mapsto j^*(K^{\wedge})$. This functor sends $\pi^*_XD(X)$ to zero. Hence, it induces a functor $E_{rh}(X) \to D^b_h(X\times \GG_m)$.
\begin{thm}\label{thm:alt1}
The functor $E_{rh}(X) \to D^b_h(X \times \GG_m)$ given by:
\[ K \mapsto j^*(K^{\wedge}) \]
is full, faithful and t-exact.
\end{thm}

\begin{proof}
The t-exactness is immediate. For full and faithful argue as follows.
Let $\cC$ denote the essential image, under Fourier transform, of $D^b_{rh}(X\times\GG_a)$ in $D^b_h(X\times \GG_a)$. 
Note that for $K\in D^b_{h}(X)$, one has a canonical isomorphism $(\pi_X^*K)^{\wedge} \simeq i_{0*}K[-1]$ (Proposition \ref{prop:wedgepull}).
Write $\bar\cC$ for the quotient of $\cC$ by $i_{0*}D^b_{rh}(X)$. Then Fourier transform gives an equivalence $E_{rh}(X) \mapright{\sim} \bar \cC$, and $j^*\colon D^b_{h}(X \times \GG_a) \to D^b_h(X\times \GG_m)$ visibly induces a functor $\bar j^* \colon \bar \cC \to D^b_h(X\times \GG_m)$. Our problem becomes to show that $\bar j^*$ is full and faithful. Let $K_1, K_2\in D^b_{rh}(X\times \GG_a)$.

To see $\bar j^*$ is full, let $\phi\colon j^*K_1^{\wedge} \to j^*K_2^{\wedge}$ be a morphism in $D^b_h(X\times \GG_m)$. The cones of the adjunction maps $K_1^{\wedge} \to j_*j^*K_1^{\wedge}$ and $K_2^{\wedge} \to j_*j^* K_2^{\wedge}$ are $i_{0*}i_0^!K_1^{\wedge}[1]$ and $i_{0*}i_0^!K_2^{\wedge}[1]$, respectively.
Hence, thanks to Lemma \ref{rhlemma}, the above adjunction maps become isomorphisms in $\bar \cC$. Consequently, in $\bar \cC$ we have a map:
\[ K_1^{\wedge} \mapright{\sim} j_*j^*K_1^{\wedge} \mapright{j_*(\phi)} j_*j^*K_2^{\wedge} \mapright{\sim} K_2^{\wedge}. \]
This map goes to $\phi$ under $\bar j^*$.

To see faithfulness, observe that a map $f\colon K_1^{\wedge} \to K_2^{\wedge}$ in $\bar \cC$ may be represented by a diagram:
\[ K_1^{\wedge} \mapleft{s} L \mapright{g} K_2^{\wedge} \]
with the cone of $s$ lying in $i_{0*}D^b_{rh}(X)$ (see \cite[Definition 2.1.11]{Neeman}). If $\bar j^*(f) = 0$, then necessarily $j^*(g) = 0$. This implies $g$ must factor through $i_{0*}i^!_0 K_2^{\wedge}$. So, by Lemma \ref{rhlemma}, $g$ is zero in $\bar \cC$. Hence, $f$ is zero.
\end{proof}

\appendix
\makeatletter
\@addtoreset{subsubsection}{section}
\makeatother
\renewcommand{\thesubsubsection}{\textbf{\thesection.\arabic{subsubsection}}}
\renewcommand{\theequation}{\bfseries\thesubsubsection.\arabic{equation}}

\section{Partial Fourier transform}\label{s:interlude}
\subsubsection{}
The partial Fourier transform is the t-exact auto-equivalence:
\[ (-)^{\wedge}\colon \Dhol(X\times \GG_a) \to \Dhol(X\times \GG_a)\]
characterized locally on smooth $X$ as follows.
Write $\GG_a = \Spec(\CC[t])$.
For $U\subset X$ an open affine and $K$ a D-module on $U\times \GG_a$, one has $\Gamma(U\times\GG_a, K^{\wedge}) = \Gamma(U\times\GG_a, K)$ as a $\Gamma(\cD_U)$-module, but with $t$ acting on $\Gamma(K^{\wedge})$ as $\partial_t$ does on $\Gamma(K)$, and $\partial_t$ acting on $\Gamma(K^{\wedge})$ as $-t$ does on $\Gamma(K)$.
The sign convention is chosen to match with that of \cite[\S7.1]{KL} and \cite[\S2.10]{K}. One has:
\[ K^{\wedge\wedge} \simeq a^*K, \]
where $a\colon X \times \GG_a \to X \times \GG_a$ is given by $a(x,t) = (x,-t)$.
Furthermore, there is a canonical isomorphism \cite[Lemme 7.1.3]{KL}:
\begin{equation}\label{eq:wedgedual} \DD(K^{\wedge}) \simeq a^*((\DD K)^{\wedge}).\end{equation}
The transform $(-)^{\wedge}$ preserves holonomicity, but not regular singularities.

\subsubsection{}The transform $(-)^{\wedge}$ is a familiar one in terms of the formalism of \S\ref{s:laplace}.
Namely, consider the diagram:
\[\xymatrix{
& X \times \GG_a \times \GG_a\ar[ld]_{p_{12}}\ar[rd]^{p_{13}} \ar[rr]^{m\colon (x,y,z)\mapsto yz} & & \GG_a \\
X \times \GG_a & & X \times \GG_a
}\]
where $p_{12}, p_{13}$ are projections onto the indicated factors. Recall from \S\ref{s:defL} that $\cL_{\lambda} \in D^b_h(\GG_a), \lambda\in \CC$, is defined by requiring $\cL_{\lambda}[1]$ to be the D-module on $\GG_a=\Spec(\CC[t])$ given by:
\[\Gamma(\cL_{\lambda}[1]) = \CC[t], \qquad \partial_t\cdot p(t) = p'(t) -\lambda p(t). \]
\begin{lemma}[{\cite[Lemme 7.1.4]{KL}}]\label{wedgelemma}
We have a canonical isomorphism:
\[ K^{\wedge} \simeq \DD(p_{13!}(p_{12}^*\DD K\otimes m^*\cL_{-1})[1]).\] 
\end{lemma}

\begin{proof}
We just need to unwind the notation in \cite[Lemme 7.1.4]{KL}. According to that result (note: our conventions for D-module functors match those of \cite{KL}, and our $\cL_{1}$ is their $L$):
\[ K^{\wedge} \simeq p_{13*}(p_{12}^!K \tilde\otimes m^!\cL_1)[1], \]
where $(-)\tilde\otimes (-) = \DD(\DD(-)\otimes\DD(-))$.
Thus,
\begin{align*}
K^{\wedge} &\simeq p_{13*}\DD(\DD(p_{12}^!K) \otimes \DD(m^!\cL_1))[1] \\
&\simeq \DD(p_{13!}(p_{12}^*\DD K \otimes m^*\DD \cL_1))[1] \\
&\simeq \DD(p_{13!}(p_{12}^*\DD K \otimes m^*\cL_{-1}[2]))[1] & \text{(by \eqref{eq:dualofclambda})} \\
&\simeq \DD(p_{13!}(p_{12}^*\DD K \otimes m^*\cL_{-1})[1]). 
\end{align*}
\end{proof}

\subsubsection{}
For $\lambda\in \CC$, recall that $i_{\lambda}\colon X \to X\times \GG_a$ denotes the map given by:
\[ i_{\lambda}(x) = (x,\lambda).\]

\begin{prop}\label{prop:wedgereal}
Let $K\in \Dhol(X\times \GG_a)$. Then there is a canonical isomorphism:
\[ i_{\lambda}^! K^{\wedge}[1] \simeq \DD(\pi_{X!}(\DD K \otimes \mu^* \cL_{-\lambda})) \]
\end{prop}

\begin{proof}Let $\tilde i_{\lambda}\colon X\times \GG_a\to X \times\GG_a\times\GG_a$ denote the map $(x,t) \mapsto (x,t,\lambda)$, and write $[\lambda]\colon \GG_a \to \GG_a$ for the map $t\mapsto \lambda t$. Then:
\begin{align*}
i_{\lambda}^! K^{\wedge}[1] &\simeq i_{\lambda}^! \DD (p_{13!}(p_{12}^*\DD K \otimes m^*\cL_{-1})) & \text{(Lemma \ref{wedgelemma})} \\
&\simeq \DD(i_{\lambda}^*p_{13!}(p_{12}^* \DD K \otimes m^* \cL_{-1})) & \text{($i_{\lambda}^! \circ \DD \simeq \DD \circ i_{\lambda}^*$)}\\
&\simeq \DD(\pi_{X!}\tilde i_{\lambda}^*(p_{12}^* \DD K \otimes m^* \cL_{-1})) & \text{(proper base change)} \\
&\simeq \DD(\pi_{X!}(\DD K \otimes \mu^*[\lambda]^*\cL_{-1})) & \text{($p_{12}\circ \tilde i_{\lambda} = \id_{X\times\GG_a}$ and $m \circ \tilde i_{\lambda}= [\lambda]\circ \mu$)}\\
&\simeq \DD(\pi_{X!}(\DD K \otimes \mu^*\cL_{-\lambda})) & \text{($[\lambda]^*\cL_{-1}\simeq \cL_{-\lambda}$)}.
\end{align*}
\end{proof}

\begin{prop}\label{prop:wedgepush}
For $K\in D^b_h(X\times \GG_a)$, we have a canonical isomorphism:
\[ i_{\lambda}^*K^{\wedge}[-1] \simeq \pi_{X!}(K\otimes \mu^*\cL_{\lambda}).\]
\end{prop}

\begin{proof}We have:
\begin{align*}
\pi_{X!}(K \otimes \mu^*\cL_{\lambda}) &\simeq \pi_{X!} (\DD\DD K \otimes \mu^*\cL_{\lambda}) & \text{(as $\DD\circ \DD \simeq \id$ on $\Dhol(X\times \GG_a)$)}\\
&\simeq \DD(i_{-\lambda}^!(\DD K)^{\wedge})[-1] & \text{(Proposition \ref{prop:wedgereal})}\\
&\simeq \DD(i_{-\lambda}^!a^*(\DD(K^{\wedge}))) [-1] & \text{(by \eqref{eq:wedgedual})}\\
&\simeq \DD(i_{\lambda}^!\DD(K^{\wedge}))[-1] & \text{($a^*=a^!$ and $a\circ i_{-\lambda} = i_{\lambda}$)}\\
&\simeq i_{\lambda}^*K^{\wedge}[-1].
\end{align*}
\end{proof}

\begin{prop}\label{prop:wedgepull}
For $K\in D^b_h(X)$, we have a canonical isomorphism:\[ i_{0*}K[-1] \simeq (\pi_{X}^*K)^{\wedge}.\]
\end{prop}

\begin{proof}
This can be proved directly from the definition of $(-)^{\wedge}$. Alternatively, observe that $i_{0*}(-)[-1]$ is the right adjoint to $i_0^*(-)[1]$. But by Proposition \ref{prop:wedgepush}, the latter is isomorphic to $\pi_{X!}(-)^{\wedge}[2]$ and this has $(\pi_X^!(-))^{\wedge}[-2] \simeq \pi_X^*(-)^{\wedge}$ as right adjoint. So we are done by the uniqueness of adjoints (Yoneda lemma).
\end{proof}

\begin{prop}\label{prop:weaknonchar}
Let $\lambda\in \CC^{\times}$, and let $K$ be a regular holonomic D-module on $X\times \GG_a$. Then $i_{\lambda}^!K^{\wedge}[1]$ is a holonomic D-module (i.e., its cohomology sheaves are concentrated in degree $0$).
\end{prop}

\begin{proof}The claim is local on $X$, so we may assume $X$ is affine. Embedding $X$ into a smooth variety ($\AA^n$, say), we may further assume $X$ is smooth and identify $K$ and $K^{\wedge}$ with their global sections.
Let $t$ denote the coordinate for $\GG_a$. Then $i_{\lambda}^!K^{\wedge}[1]$ is represented by the complex $[K^{\wedge} \mapright{t-\lambda} K^{\wedge}]$ concentrated in degrees $0$ and $-1$. Under the identification with global sections, the action of $t$ on $K^{\wedge}$ becomes the action of $\partial_t$ on $K$. Hence, we need to show $\partial_t - \lambda$ is injective as an operator on $K$. Let 
\[ T = \{v\in K \mid t^n v = 0 \text{ for some $n$}\}. \]
Then we have an exact sequence:
\[ 0\to T \to K \to K[t^{-1}], \]
where $K[t^{-1}]$ denotes the localization of $K$ at $t$.
So it suffices to show $\partial_t-\lambda$ is injective on $T$ and $K[t^{-1}]$.

Let $v\in T$ be non-zero and set $d=\max\{k \mid t^k v \neq 0\}$. Then:
\[ t^{d+1}(\partial_t - \lambda)v = \partial_tt^{d+1}v - (d+1)t^dv - \lambda t^{d+1}v = -(d+1)t^d v \neq 0. \]
I.e., $\partial_t-\lambda$ is injective on $T$.

For $K[t^{-1}]$, set $u = t^{-1}$ and $\partial_u = -t^2\partial_t$, so that $[\partial_u, u]=1$ and $\partial_t = -u^2\partial_u$. Write $M$ for $K[t^{-1}]$ regarded as a $\Gamma(\cD_{X\times \Spec(\CC[u])})$-module. 
Then our problem becomes: show $-u^2\partial_u - \lambda$ is injective on $M$. As $K$ is regular holonomic, so is $M$. Since $M$ is holonomic, we have the Kashiwara-Malgrange $V$-filtration \cite{Kas} at $u =0$. This is an exhaustive, decreasing filtration $V^{\alpha}M$, indexed by $\CC$, such that:
\begin{enumerate}
\item $u\cdot V^{\alpha}M \subset V^{\alpha+1}M$;
\item $\partial_u \cdot V^{\alpha}M \subset V^{\alpha - 1}M$.
\end{enumerate}
Thus,
\[ -u^2\partial_u\cdot V^{\alpha}M \subset V^{\alpha +1}M.\]
Let $v\in \ker(-u^2\partial_u - \lambda)$. As the $V$-filtration is exhaustive, $v\in V^{\alpha}M$ for some $\alpha$.
Then, as $\lambda \in \CC^{\times}$, we must have $v\in V^{\alpha + 1}M$.
Consequently, $v\in \bigcap_{k\in\ZZ_{\geq 0}} V^{\alpha +k}M$.
But as $M$ is regular, the $V$-filtration is also separated.\footnote{This can fail for irregular D-modules. See \cite[Example 2.10.22]{Bj}.} Consequently, $v=0$.
\end{proof}

\section{Forget supports isomorphism}\label{s:forgetsupports}
\subsubsection{}
This appendix gives an independent proof of a stronger version of Proposition \ref{prop:weaknonchar}. This stronger result is used only in Theorem \ref{thm:commutesoperations}, to show that realizations commute with Verdier duality and the operations $f^!, f_*$. 

\subsubsection{}We will write $\Spec(\CC[u])\subset \PP^1$ for the standard affine chart at $\infty$, glued to $\GG_a = \Spec(\CC[t])$ along $\Spec(\CC[t,t^{-1}])$ by $u=t^{-1}$. So $\partial_u = -t^2\partial_t$ on the overlap. For a variety $X$, we write $j\colon X \times \GG_a \to X \times \PP^1$ for the evident inclusion. Note that $j$ is an affine open immersion, hence $j_*$ is t-exact. Write $i\colon X \times \{\infty\} \to X \times \PP^1$ for the complement.

\begin{lemma}
\label{lem:coherence}
Let $\lambda\in \CC^{\times}$, and let $K$ be a regular holonomic module on $X\times \GG_a$, with $X$ smooth and affine. 
Set $M = \Gamma(K\otimes \mu^*\cL_{\lambda})$.
Write $M[t^{-1}]$ for the localization of $M$ at $t$.
Let $\cA \subset \Gamma(\cD_{X\times\Spec(\CC[u])})$ be the subalgebra generated by $\Gamma(\cD_X), u$ and $u\partial_u$. 
Then:
\begin{enumerate}
\item $M[t^{-1}] \simeq \Gamma(X \times \Spec(\CC[u]), j_*(K\otimes \mu^*\cL_{\lambda}))$;
\item $M[t^{-1}]$ is finitely generated over $\cA$.
\end{enumerate}
\end{lemma}

\begin{proof}(i) is standard and essentially the corresponding statement for quasi-coherent sheaves in algebraic geometry: as $j$ is an affine open immersion, $j_*$ on D-modules is the ordinary (underived) sheaf theoretic pushforward. In other words,
\[ \Gamma(X \times \Spec(\CC[u]), j_*(K\otimes \mu^*\cL_{\lambda})) = \Gamma(X \times \Spec(\CC[t,t^{-1}]), K \otimes \mu^*\cL_{\lambda}).\]
As $X\times \Spec(\CC[t,t^{-1}])$ is the principal open $\{t\neq 0\}$ of the affine variety $X\times \GG_a$, this is $M[t^{-1}]$, compatibly with differential operators since derivations extend uniquely to localizations.

For (ii), we may identify $M=\Gamma(K\otimes \mu^*\cL_{\lambda})$ with $\Gamma(K)$. Under this identification the $\partial_t$ action on $\Gamma(K\otimes \mu^*\cL_{\lambda})$ corresponds to the action of $\partial_t - \lambda$ on $\Gamma(K)$. The action of $\Gamma(\cD_X)$ and $\cO_{X\times \GG_a}$ is unchanged. 
To disambiguate, denote the $\Gamma(\cD_{X\times \GG_a})$-action on $\Gamma(K)$ corresponding to that on $M$ by $\star$, and write $\cdot$ for the usual action.
So we may identify $M[t^{-1}]$ and the localization $\Gamma(K)[t^{-1}]$, with:
\[ u\partial_u \star v = (u\partial_u + \lambda u^{-1})\cdot v, \qquad v\in \Gamma(K)[t^{-1}]. \] 
As $K$ is regular holonomic, by \cite[Proposition 2.14.4]{G}, the module $\Gamma(K)[t^{-1}]$ admits a finitely generated $\Gamma(\cD_X)[u]$-submodule $L\subset \Gamma(K)[t^{-1}]$, such that:
\[ \Gamma(K)[t^{-1}] = \bigcup_{k\in \ZZ_{\geq 0}} u^{-k}L, \qquad u\partial_u \cdot L \subset L.\]
We will show that $u^{-k}L \subset \cA \star L$ for all $k\in \ZZ_{\geq 0}$. This will give the result, since it will imply $M[t^{-1}] = \Gamma(K)[t^{-1}] = \cA \star L$, and $L$ is finitely generated over $\Gamma(\cD_X)[u]\subset \cA$.

Proceed by induction on $k$. The case $k=0$ is trivial. 
Assume $u^{-k}L \subset \cA \star L$.
Now, for $x\in L$:
\[ u\partial_u \star (u^{-k}x) = -ku^{-k}x + u^{-k}(u\partial_u\star x) = -ku^{-k}x + u^{-k}(u\partial_u\cdot x) + \lambda u^{-(k+1)}x.\]
Hence:
\[ \lambda u^{-(k+1)}x = u\partial_u\star (u^{-k}x) + ku^{-k}x - u^{-k}(u\partial_u\cdot x) \]
By the induction hypothesis, the first two terms on the right are in $\cA\star L$. The last term is in $u^{-k}L$, since $u\partial_u\cdot L\subset L$. 
Thus, the induction hypothesis gives that it too is in $\cA\star L$.
As $\lambda\in \CC^{\times}$, we conclude $u^{-(k+1)}x \in \cA \star L$.
\end{proof}

\begin{thm}\label{thm:forgetsupports}
Let $\lambda\in \CC^{\times}$, and let $K$ be a regular holonomic D-module on $X\times \GG_a$. Then the canonical forget supports map:
\[ \pi_!(K\otimes \mu^*\cL_{\lambda}) \to \pi_*(K\otimes \mu^*\cL_{\lambda}) \]
is an isomorphism.
\end{thm}

\begin{proof}
It suffices to show that the forget supports map:
\[ j_!(K\otimes \mu^*\cL_{\lambda}) \to j_*(K\otimes \mu^*\cL_{\lambda}) \]
is an isomorphism.
This question is local on $X$. So we can and will assume that $X$ is smooth affine and pass from D-modules to their global sections without further comment.
As $j$ is an affine open immersion, the functors $j_!$ and $j_*$ are t-exact.
Furthermore, the cone of the forget supports map above is $i_*i^*j_*(K\otimes \mu^*\cL_{\lambda})$. So, it suffices to show:
\begin{enumerate}[label=(\alph*)]
\item $j_*(K\otimes\mu^*\cL_{\lambda})$ admits no non-zero quotient supported on $X\times \{\infty\}$;
\item $j_!(K\otimes\mu^*\cL_{\lambda})$ admits no non-zero submodule supported on $X\times \{\infty\}$.
\end{enumerate}
These statements are Verdier dual in the sense that (b) for the pair $(K,\lambda)$ is (a) for the pair $(\DD K, -\lambda)$. Thus, it suffices to prove (a). 
Restriction to $X \times \Spec(\CC[u])$ is t-exact, and a module supported on $X \times \{\infty\}$ vanishes if and only if its restriction to $X\times \Spec(\CC[u])$ does.
Hence, we are reduced to showing that the restriction of $j_*(K\otimes\mu^*\cL_{\lambda})$ to $X \times \Spec(\CC[u])$ admits no non-zero quotient supported on $X\times \{\infty\}$.

Let $M=\Gamma(K\otimes \mu^*\cL_{\lambda})$. By Lemma \ref{lem:coherence}(i), we have $\Gamma(X \times \Spec(\CC[u]), j_*(K\otimes \mu^*\cL_{\lambda})) \simeq M[t^{-1}]$. Write $Q$ for the sections, over $X\times \Spec(\CC[u])$, of a quotient as above that is supported on $X\times \{\infty\}$.
Then we have a surjection:
\[ \varphi\colon M[t^{-1}] \to Q, \]
and:
\begin{enumerate}
\item $u$ acts invertibly on $M[t^{-1}]$;
\item every element of $Q$ is annihilated by some power of $u$, since $Q$ is supported on $X\times \{\infty\}$;
\item the algebra $\cA$ of Lemma \ref{lem:coherence} preserves $\ker(u^k\colon Q\to Q)$ for each $k$ because $\Gamma(\cD_X)$ and $u$ commute with $u^k$, while:
\[ u^k(u\partial_u \cdot v) = u\partial_u(u^k \cdot v) - ku^k\cdot v.\]
\end{enumerate}
By Lemma \ref{lem:coherence}(ii), there exists a finite set $\{v_i\}_i$ of generators for $M[t^{-1}]$ over $\cA$. By (ii) above, there is an $n$ such that $u^n\varphi(v_i) = 0$ for all $i$. Hence:
\[ Q = \varphi(M[t^{-1}]) = \sum_i \cA \cdot \varphi(v_i). \]
So by (iii), $Q\subset \ker(u^n\colon Q \to Q)$. But by (i), $uQ = \varphi(uM[t^{-1}]) = \varphi(M[t^{-1}]) = Q$. So $u^nQ = Q$, and $Q=0$.
\end{proof}

\subsubsection{}\label{s:mochizuki}
Theorem \ref{thm:forgetsupports} yields a canonical isomorphism:
\[ i_{\lambda}^* K^{\wedge}[-1] \mapright{\sim} i_{\lambda}^!K^{\wedge}[1], \]
for $\lambda\in \CC^{\times}$ and $K\in D^b_{rh}(X\times \GG_a)$.
In fact, a more refined microlocal statement is true. Namely, for $\lambda\in \CC^{\times}$ and $K$ regular holonomic on $X\times \GG_a$, the map $i_{\lambda}$ is non-characteristic for $K^{\wedge}$. This is due to T. Mochizuki (private communication).

\begin{remark}
Although a bit obscured by the language, the proof of Theorem \ref{thm:forgetsupports} essentially amounts to showing that the tame nearby cycles at $\infty$ of $K\otimes \mu^*\cL_{\lambda}$ vanish under the given hypotheses.
These statements fail without the regularity assumption: $K=\cL_{-\lambda}$ for $X=\Spec(\CC)$ provides a counterexample. For this $K$ the exponential irregularity at $\infty$ is cancelled: $\cL_{-\lambda}\otimes \cL_{\lambda} \simeq \cO_{\GG_a}$. The resulting object is clearly regular/tame at $\infty$, and the forget supports isomorphism fails.
The proof of the corresponding statement for $\ell$-adic sheaves over finite fields is pleasantly transparent. Informally, Artin-Schreier covers are so horribly ramified at $\infty$ that no tensoring with a tame sheaf can unwind them. The action of the inertia group at $\infty$ is totally wild.
\end{remark}

\end{document}